\documentclass[12pt]{article}
\usepackage{latexsym,amssymb,amsmath,amsfonts,amsthm,xcolor}
\usepackage{graphicx}
\newtheorem{theorem}{Theorem}

\newtheorem{lemma}{Lemma}[section]

\theoremstyle{definition}

\def\beq{ \begin{equation} }
\def\eeq{ \end{equation} }
\def\mn{\medskip\noindent}

\def\ep{\varepsilon}

\def\square{\vcenter{\vbox{\hrule height .4pt
  \hbox{\vrule width .4pt height 5pt \kern 5pt
        \vrule width .4pt} \hrule height .4pt}}}

\def\TT{{\cal T}}

\def\RR{\mathbb{R}}
\def\ZZ{\mathbb{Z}}

\def\var{\hbox{var}\,}
\def\sqz{\kern -0.2em}

\def\clearp{}

\graphicspath{%
    {converted_graphics/}% inserted by PCTeX
    {/}% inserted by PCTeX
}
\begin{document}

\title{The Zealot Voter Model}
\author{Ran Huo and Rick Durrett \thanks{Both authors have been  partially supported by NSF grant DMS 1505215 from the probability program.}
\\
Dept. of Math, Duke U., \\ \small P.O. Box 90320, Durham, NC 27708-0320}
\date{\today}						

\maketitle

\begin{abstract}
Inspired by the spread of discontent as in the 2016 presidential election, we consider a voter model
in which 0's are ordinary voters and 1's are zealots. Thinking of a social network, but desiring the simplicity 
of an infinite object that can have  a nontrivial stationary distribution, 
space is represented by a tree. The dynamics are a variant of the biased voter: if $x$ has degree $d(x)$ then
at rate $d(x)p_k$ the individual at $x$ consults $k\ge 1$ neighbors.
If at least one neighbor is 1, they adopt state 1, otherwise they become 0. In addition at rate $p_0$
individuals with opinion 1 change to 0. As in the contact process on trees, we are interested in determining when the zealots
survive and when they will survive locally. 
\end{abstract}
 
\maketitle

\section{Introduction}

In the standard (linear) voter model, which was introduced by Holley and Liggett \cite{HL}, a site flips at a rate equal to the fraction of
neighbors that have the other opinion. Cox and Durrett \cite{CDNLV} began the study of voter models with non-linear flip rates. One of the
most successful ideas from that paper is the threshold-$\theta$ voter model in which sites flip at rate 1 if at least $\theta$ neighbors have
the opposite opinion. Liggett \cite{L94} obtained results for coexistence of opinions when $\theta=1$, while Chatterjee and Durrett \cite{thresh2} showed that the model with $\theta\ge2$ had a discontinuous phase transition on the random $r$-regular graph when $r \ge 3$. 
Lambiotte and Redner \cite{vacv} studied the ``vacillating voter model'' in which a voter looks at the opinions of two randomly chosen neighbors and flips if
at least one disagrees. At about the same time, Sturm and Swart considered ``rebellious voter models'' in one dimension. In the one-sided case $\xi_t(i)$ canges its opinion at rate $\alpha$ if $\xi_t(i+1)\neq \xi_t(i))$ and at an additional rate $1-\alpha$ if $\xi_t(i+1)\neq \xi_t(i+2))$. They also considered a spatially symmetric version. In all these variants of the voter model, the process is symmetric under interchange of 0's and 1's. Our zealot voter model does not have that symmetry.

In our process, space is represented by a tree $\mathcal{T}$ in which the degree of each vertex $x$ satisfies $3 \le d_{min} \le d(x) \le M$.
This guarantees that our trees are infinite.  Voters can be in state 0 (ordinary voter) or 1 (zealot). Given a probability distribution $p_k$ on $\{ 0, 1, 2, \ldots d_{min} \}$, if $k \ge 1$ then at rate $d(x)p_k$ the voter $x$ picks $k$ neighbors without replacement. As in the vacillating voter model the voter becomes 1 if at least one of the chosen neighbors is a 1, otherwise it becomes 0. In addition at rate $p_0$, voters change their opinion from 1 to 0. 

If $p_0=0$ then this model is a variant of the biased voter model. In that system, a 0 at $x$ changes to 1 at rate $\lambda$ times $n_1(x)$ the number of neighbors of $x$ that are in state $1$, and a 1 at $x$ changes to 0 at rate $n_0(x)$ the number of neighbors of $x$ that are in state $0$. If the degree is constant then the behavior of the process is easy to understand. If we start from finitely many 1's then the number of 1's at time $t$, $N^1_t$ decreases by 1 at a rate equal to $D_t$ the number of $(1,0)$ edges, and increases by 1 at rate $\lambda D_t$. Thus $N^1_t$ is a time change of  a simple random walk that increases by 1 with probability $\lambda/(\lambda+1)$ and decreases by 1 with probability $1/(\lambda+1)$. Using this observation it is easy to show that the critical value for the survival of 1's $\lambda_c=1$. In our setting sites do not have constant degree and we have a different type of bias. This makes things more complicated, and it is hard to get precise results on the location of phase transitions.

Our process is additive in the sense of Harris \cite{H76} and hence can be constructed on a graphical representation with independent 
Poisson processes $T^{x,i}_n$, $n \ge 1$, $0 \le i\le d_{min}$. 

\begin{itemize}
  \item 
The $T^{x,0}_n$ have rate $p_0$. At these times we write a $\delta$ at $x$ that will kill a 1 at the site.

\item
The $T^{x,i}_n$ have rate $d(x)p_i$. At time $T^{x,i}_N$ we write a $\delta$ at $x$ that will kill a 1 at the site.
In addition we draw oriented arrows to $x$ from $i$ neighbors $y_1, \ldots y_i$ chosen at random without replacement
from the set of neighbors. If any of the $y_i$ are in state 1, then $x$ will be in state 1. Otherwise it will be in state 0.
\end{itemize}

We will often use coordinate notation for the process, i.e., $\xi_t(x)$ gives the state of $x$ at time $t$. However it is also convenient to use the set-valued approach with $\xi_t^A$ giving the set of sites occupied by zealots at time $t$ when the initial set of zealots is $A$. 
Intuitively, the process $\xi^A_t$ can be defined by introducing fluid at the sites in $A$. The fluid flows up the graphical representation, being blocked by $\delta$'s, and flowing across edges in the direction of their orientations. The state at time $t$, $\xi^A_t$ is the set of points that can be reached by fluid at time $t$ starting from some site in $A$ at time $0$.

A nice feature of this construction is that it allows us to define a dual process in which fluid flows down the graphical representation, is blocked by $\delta$'s and flows across edges in a direction opposite their orientations. We let $\zeta^{B,t}_s$ be the points reachable at time $t-s$ starting from $B$ at time $t$. {
It is immediate from the construction that
\begin{equation}\label{dual1}
\{ \xi_t^A \cap B \neq \emptyset \} = \{A \cap \zeta^{B,t}_t  \neq \emptyset \}
\end{equation}
It should be clear from the construction that the distribution of  $\zeta^{B,t}_s$ for $0 \le s \le t$ does not depend on $t$,
so we drop the $t$ and write the duality as
\begin{equation}\label{dual1}
P( \xi_t^A \cap B \neq \emptyset ) = P( A \cap \zeta^{B}_t  \neq \emptyset )
\end{equation}
The dual $\zeta^B_t$ is a coalescing branching random walk (COBRA) with the following rules. A particle at $x$ dies at rate $p_0 $ and 
at rate $d(x)p_k$ it dies after giving birth to offspring that occupy $k$ of the neighboring sites chosen at random without replacement.  
For more details see Griffeath \cite{G78}. 

In the case $p_0=0$ this pair of dual processes has been studied by Cooper, Radzik, and Rivera \cite{CRR}. In their situation the zealot voter model 
is called a biased infection with a persistent source (BIPS). The phrase persistent source refers to the fact that the BIPS model has one individual that stays infected forever. Their main interest is in the cover time for COBRA, i.e., the time for the process to visit all of the sites. By duality this is related to time for the BIPS to reach all 1's.

In this paper, when we say that a process {\it survives} we mean that with positive probability it avoids becoming $\emptyset$.
We say a process {\it survives locally} if with positive probability the root $0$ is occupied infinitely many times.

When $A=B=\{0\}$, (\ref{dual1}) implies 
\beq
P( 0 \in \xi^0_t ) = P( 0 \in \zeta^0_t )
\label{localsurv}
\eeq
so local survival of one process implies local survival of the other. Taking one of the sets $={\cal T}$ and the other $=\{0\}$ we get 
$$
P( \xi^0_t  \neq \emptyset ) = P( 0 \in \zeta^{\cal T}_t ) \qquad P( \zeta^0_t  \neq \emptyset ) = P( 0 \in \xi^{\cal T}_t ) 
$$
so survival of one process implies that the other has a nontrivial stationary distribution obtained by letting $t\to\infty$ in 
$\zeta^{\cal T}_t$ or $\xi^{\cal T}_t$. Our first result is very general. 

\begin{theorem}\label{genthm}
On any tree with degrees $3 \le d(x) \le M$, the zealot voter model survives if 
$$
\sum_{k\geq 2}(k-1)p_k- p_0>0.
$$
\end{theorem}

\noindent
The result is proved by comparing the growth of the process at the ``frontier'' with a branching process. For the definition of frontier, see the text before Lemma \ref{lem1}. Note that the degree distribution does not appear in the condition.

\subsection{Results for $d-$regular Trees}\label{d-trintro}

Let $\beta = 1 - (d-1)^{-2}$ be the probability that two independent random walks on the $d$-regular tree that start at distance two never hit.
See Lemma \ref{nohit} for a proof of this.

\begin{theorem} \label{ddieA}
On a d-regular tree the COBRA dies out if 
\beq
d\beta\sum_{k\geq 2}(k-1)p_k-p_0<0.
\label{cdie}
\eeq
When this holds the zealot voter model does not have a nontrivial stationary distribution.
\end{theorem}

\noindent
To explain the condition, note that in the dual, a particle dies at rate $p_0$ and gives birth to $k$ particles at rate $dp_k$. To get an upper bound on the growth of the dual (i) we ignore coalescence between individuals that are not siblings, and (ii) if $k$ particles are born we number
them $1, 2, \ldots k$ and ignore coalescence between particles $i>1$ and $j>1$. This gives an upper bound on the dual COBRA.

\begin{theorem}\label{ddieB}
If \eqref{cdie} holds then the zealot voter model dies out on a d-regular tree.
\end{theorem}

\begin{proof} 
Theorem \ref{ddieA} is proved by showing the expected number of particles in the COBRA, denoted as $E|\zeta^0_t|$, converges to $0$ as $t\to \infty$. By symmetry,
$$
E|\zeta^0_t|=\sum_x P(x\in \zeta^0_t)=\sum_x P(0\in \zeta^x_t)\ge P(0\in\zeta^1_t)
$$ 
where $\zeta^1_t$ is the COBRA starting with all sites occupied. The last property follows from the additivity of processes constructed on a graphical representation, i.e., if $A = \cup_i A_i$, a finite or infinite union, then
$$
\xi_t^A = \cup_i \xi_t^{A_i}
$$
This implies that if \eqref{cdie} holds then COBRA has no stationary distribution, and by duality the zealot voter model dies out.
\end{proof}

To study the local survival of our voter model, we use \eqref{localsurv}
to change the problem to studying the local survival of the COBRA. 
Let $\mu = \sum_k kp_k$ is the mean number of offspring in the dual process

\begin{theorem}\label{ddlocal}
Given a d-regular tree $T$, the zealot voter model dies out locally if 
$$
\mu<\frac{d(1-p_0)+p_0}{2\sqrt{d-1}}.
$$
If $p_0=0$ this is $\mu < d/(2\sqrt{d-1})$.
\end{theorem}

\noindent
This result is proved by comparing COBRA with a branching random walk by ignoring coalescence. The second bound is sharp for the branching random walk with no death. That is, the corresponding branching random walk visits the root with positive probability if $\mu> d/(2\sqrt{d-1})$ and that the root is visited finitely many times if $\mu < d/(2\sqrt{d-1})$. This result can be found in Pemantle and Stacey \cite{PemSta}. There they studied the branching random walk on trees where each particle gives birth at a rate $\lambda$ independently onto each neighbor, and  dies at rate 1. Since our branching process has simultaneous births and deaths we modify their proof to cover our situation and give the proof in Lemma \ref{LSbrw}.

To give sufficient conditions for local survival, we follow a tagged particle in the COBRA. If there is a particle produced on the site closer to the root, we follow this particle; otherwise we follow a new particle chosen uniformly at random from the offspring and ignore the rest. The recurrence of the tagged particle implies the local survival of COBRA. Using this idea leads to a simple proof of a condition for local survival, but the result is not very accurate.

\begin{theorem}\label{qpf}
 On a $d$-regular tree the zealot voter model survives locally if $p_0=0$ and $\mu > d/2$.
\end{theorem}

\begin{proof}
Note that if $i$ is the number of particles produced in a branching event and $q_i$
is the probability all of them going further from to the root then
$$
q_i = \frac{ \binom{d-1}{k} }{ \binom {d}{k} } = \frac{ (d-1)!}{ k! (d-1-k)! } \cdot \frac{k! (d-k)!} {k!} = \frac{d-k}{d}
$$
Thus if we follow the particle that gets closer to the root then it jumps by $-1$ with probability
$$
\sum_k p_k \frac{k}{d} = \frac{\mu}{d}
$$
and the tagged particle will be positive recurrent if $\mu > d/2$.
\end{proof}

Our next Theorem, which uses some ideas from the proof of Lemma 4.57 in Liggett's 1999 book \cite{L99},  gives a more precise result.

\begin{theorem}\label{dslocalbbd}
On a $d$-regular tree the zealot voter model survives locally if $p_0=0$ and 
$$
\mu>\frac{d}{\sqrt{d-1}+1}.
$$
\end{theorem}

\noindent
Combining this with Theorem \ref{ddlocal}, we notice that when $p_0=0$ the phase transition of local survival $\mu_l$ satisfies
$$  
\mu_l\in\left[\frac{d}{2\sqrt{d-1}},\ \frac{d}{1+\sqrt{d-1}}\right]
$$

\subsection{Results for Galton-Watson Trees}

In a Galton-Watson process with $Z_0=1$ each individual in generation $n$ has an independent and identically distributed number of children,
which are members of generation $n+1$.
The Galton-Watson tree is the genealogy of this process. The one member of generation 0 is the root.
Edges are drawn from each individual in generation $n$ to their children.
Let $p_k$ be the probability of $k$ children. We have assumed $p_k=0$ unless $3 \le d_{min} \le k \le M$, so the tree is infinite with 
probability 1, and all vertices have at most $M$ children. 

To prove an analogue of Theorem \ref{ddieA}, we formulate our model as a voter model perturbation: let $\bar p_i = \ep p_i$ when $i \neq  1$ 
and choose $\bar p_{1}$ to make the $\bar p_i$ sum to 1. A random walk that jumps to each neighbor at rate 1 
has a reversible stationary distribution that is uniform on
the graph. Let $\pi_m$ be the fraction of vertices in the tree with degree $m$, and let $\mu_{m,k}$ be the expected number of surviving particles in the dual
when we pick $k$ neighbors of a vertex of degree $m$ at random and run the coalescing random walk to time $\infty$.

\begin{theorem}\label{prtbxn} Let $\delta>0$. If $\ep$ is small then the COBRA dies out if 
$$
\sum_m \pi_m \sum_k kp_k (\mu_{m,k}-1) - p_0 < -\delta
$$
and survives if the last quantity is $>\delta$.
\end{theorem}  

\noindent
This result can be easily proved using the techniques in \cite{LV}. The key idea is that when $\ep$ is small most of the
steps in the dual are random walk steps, and the random walk is transient, so any coalescence occurs soon after branching, and the dual is essentially
a coalescing branching random walk. These ideas go back to \cite{CDP}, where they were used on $\ZZ^d$ with $d\ge 3$. 
More recent applications include \cite{CDT, LV, IGF2}. The zealot voter model has an additive dual, so things are simpler,
and we can use the approach of \cite{DZ}.  In Section 4 we will provide more details about the method.

\mn
{\bf Remark.} The last result concerns the survival of the dual, which is the same as the existence of a nontrivial stationary distribution
for the zealot voter model.

\medskip
Our next result concerns local survival. Given any Galton-Watson tree $\mathcal{T}^{GW}$, let $M$ denote its maximal degree,
and let ${\cal T}^M$ be the tree in which each vertex has $M$ children. Let $\mu_l(G)$ denote the threshold for local survival of the COBRA on graph $G$. Note the expected number of new born particles at each time are the same on both trees. Since particles on tree $\mathcal{T}^M$ have more tendency to move further away from the root, a simple comparison leads to 
$$
\mu_l\left(\mathcal{T}^{GW}(\eta_t)\right)\le \mu_l\left(\mathcal{T}^M(\eta_t)\right)
$$
where $\eta_t$ is the BRW without coalescence. The comments under Theorem \ref{ddlocal} says for $p_0=0$,
$$
\mu_l\left(\mathcal{T}^M(\eta_t)\right)=M/(2\sqrt{M-1})
$$   
It follows immediately that
\begin{theorem}\label{lcdout}
If $p_0=0$ and $\mu<M/(2\sqrt{M-1})$ then COBRA and the zealot voter model both die out locally.
\end{theorem}
Next we look for conditions implying local survival. On a tree we define the level $\ell_x$ of a vertex $x$ to be its distance to the root.
As on $d-$regular trees, our strategy is to follow a tagged particle and seek conditions guaranteeing its recurrence.  
Let $X_t$ be the level of the tagged particle at time $t$. If $\phi$ is a harmonic function for the tagged particle process $X_t$, 
i.e. $\phi(X_t)$ is a martingale, then it follows from the optional stopping theorem that If $T_0$ is the time to hit the root and $T_N$ is the first time the walk hits a site at level $N$
\begin{equation}\label{har}
\phi(1)\ge\left( \displaystyle \min_{x: \ l_x=N}\phi(x)\right)P_1\left(T_N<T_0\right)
\end{equation}
where the subscript $1$ on $P$ indicates that $X_0$ is at level 1. From \eqref{har}  we see that if $\phi(x)$ goes to $\infty$ along all paths to $\infty$ in the tree, then the tagged particles is recurrent. In order for $\phi$ to be a harmonic function
\begin{equation*}
\phi(x+1)-\phi(x)=\frac{p_x}{1-p_x}\left[ \phi(x)-\phi(x-1)\right]=\frac{\mu}{d(x)-\mu}\left[ \phi(x)-\phi(x-1)\right]
\end{equation*}
where $p_x = \mu/d(x)$ is the probability the tagged particle moves closer to the root. 
Taking logarithms, then this is
\begin{equation*}
\log\left[\phi(x+1)-\phi(x)\right]=\log\left[ \phi(x)-\phi(x-1)\right]+\log\left[ \frac{\mu}{d(x)-\mu}\right]
\end{equation*}

As we will now explain, there is a natural mapping from the $\log$-increments of the harmonic function to a branching random walk on $\mathbb{R}$. If we consider a particle at level $x$ to be at $\log\left[ \phi(x)-\phi(x-1)\right]$ on $\RR$ then $d(x)-1$ new particles will be dispersed to 
$$
\log\left[ \phi(x)-\phi(x-1)\right]+\log\left[ \frac{\mu}{d(x)-\mu}\right].
$$
As a result along any genealogical path, the distance between  two consecutive generations  is i.i.d~with law the same as $\log[\mu/(d(x)-\mu)]$. 

This process just described is different from the usual branching random walk in which children are dispersed independently from their parent. However Biggins \cite{B77} has proved results for more general branching random walks that contain ours as a special case. Let $F(t) = E( \zeta(-\infty,t])$ be the expected number of children that lie in $(-\infty,t]$ and define 
the Laplace transform of the mean measure by 
$$
m(\theta) = \int e^{-\theta t} \, dF(t)
$$

\begin{theorem} \label{GWls}
If $\min_{\theta\ge 0}m(\theta)<1$ then the leftmost particle in the branching random walk goes to $\infty$. This implies $\phi$ goes to $\infty$ along all paths to $\infty$ in the tree and we have local survival.
\end{theorem}

To apply this result to our examples, we begin by noting that 
\begin{equation*}
m(\theta)=\sum_{j\ge 3}q_j(j-1)\left(\frac{j-\mu}{\mu}\right)^\theta
\end{equation*}
It is not easy to use this formula with Theorem \ref{GWls} to get explicit predictions, so we focus on Galton-Watson tree with degrees only 3 and 4. Let $\mu=3q_3+4q_4$ and 
$$
\nu(0)=\min_{\theta\ge 0}m(\theta). 
$$
We have computed the threshold for various $\mu$ in Section \ref{3n4}. See also Figure \ref{fig:nu(0)1}.

\begin{figure}[h] % float placement: (h)ere, page (t)op, page (b)ottom, other (p)age
  \centering
  % file name: C:/Users/rtd/Dropbox/Papers/ZVM/nu(0)new.pdf
  \includegraphics[bb=53 58 737 555,width=3.83in,height=2.78in,keepaspectratio]{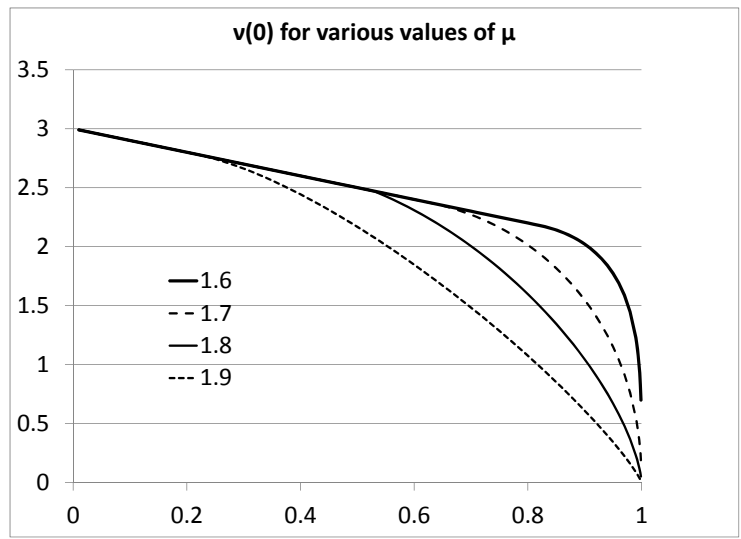}
  \caption{ $\nu(0)$ as a function of $q_3$. Local survival occurs when $q_3\ge0.996$ for $\mu=1.6$; $q_3\ge0.97$ for $\mu=1.7$; $q_3\ge0.91$ for $\mu=1.8$; and $q_3\ge 0.82$ for $\mu=1.9$.}
  \label{fig:nu(0)1}
\end{figure}

\newpage

\section{Proof of Theorem \ref{genthm}}

There are four steps in the proof.

\begin{itemize}
\item
We begin by deriving a differential equation for the expected number of occupied sites. 
\item
We define the frontier and the external boundary of a set of occupied sites and prove lower bounds on their sizes.
\item
Combining the first two steps we obtain differential equations that lower bound the number of occu[pied sites
and the size of the frontier.
\item
We prove Theorem \ref{genthm} by showing that the set of occupied sites dominates a supercritical branching walk. 
\end{itemize} 

\subsection{Step 1: Derivation of the ODE}

Let $d_k(x) = (d(x)-1)\cdots (d(x)-(k-1))$. Note that $d_k(x)$ is the number of ways of picking $k-1$ things out of
$d(x)-1$ when the order of the choices is important. Using $x*(k-1) \neq y_k$ to indicate that we sum over all ordered
choices of $k-1$ different neighbors $y_1,...,y_{k-1}$ of $x$ that are not $\neq y_k$.
\begin{align}
\frac{d}{dt} &\sum_x P(\xi_t(x)=1)=-p_0 \sum_x P(\xi_t(x)=1)
\nonumber \\
&+ \sum_{k>1} \sum_{x,y_k\sim x} \frac{p_k}{d_k(x)}\sum_{x*(k-1)\neq y_k}\left[P\left( \xi_t(x)=1,\xi_t(y_k)=0\right)-P\left( \xi_t(x)=1,\hbox{all}\ \xi_t(y_i)=0\right)\right]
\label{ode}\\
&+ \sum_{k>1} \sum_{x,y_k\sim x} \frac{p_k}{d_k(x)}\sum_{x*(k-1)\neq y_k}P(\xi_t(x)=\xi_t(y_k)=0, \xi_t(y_i)=1 \ \hbox{for some} \ i<k)
\nonumber
\end{align}

\noindent
Note that the second and third terms are $\ge 0$.

\begin{proof}
Breaking things down according to the value of $k$, treating births and deaths separately, and noting that in the last four terms
jumps occur at rate $d(x)$
\begin{align}
\frac{d}{dt}P(\xi_t(x)=1)&=-p_0P(\xi_t(x)=1)
\nonumber\\
&-p_1\sum_{y\sim x}P(\xi_t(x)=1,\xi_t(y)=0)
\nonumber\\
&+p_1\sum_{y\sim x}P(\xi_t(x)=0,\xi_t(y)=1)
\label{step0}\\
&-\sum_{k>1}\sum_x \frac{p_k}{d_k(x)}\sum_{x*k}P\left( \xi_t(x)=1, \xi_t(y_i)=0 \hbox{ for all $1\le i\le k$}  \right)
\nonumber\\
&+\sum_{k>1}\sum_x\frac{p_k}{d_k(x)}\sum_{x*k}P\left(\xi_t(x)=0, \xi_t(y_i)=1 \hbox{ for some $1\le i\le k$} \right)
\nonumber
\end{align}
If we sum over $x$ then the second and third terms cancel. We now fix $k$ and split the last term into two
\begin{align}
&=-\sum_x \frac{p_k}{d_k(x)}\sum_{x*k}P\left( \xi_t(x)=1,\hbox{ all } \xi_t(y_i)=0\right)
\nonumber\\
&+\sum_x\frac{p_k}{d_k(x)}\sum_{x*k} P\left(\xi_t(x)=0, \xi_t(y_k)=1\right)
\label{step1}\\
&+\sum_x\frac{p_k}{d_k(x)}\sum_{x*k}P\left(\xi_t(x)=\xi_t(y_k)=0, \xi_t(y_i)=1 \hbox{ for some $1\le i< k$} \right)
\nonumber
\end{align}
Recalling the definition of $d_k(x)$, the first sum can be written as 
\begin{align*}
&\sum_{x,y_k\sim x}\frac{p_k}{d_k(x)}\sum_{x*(k-1)\neq y_k}P\left( \xi_t(x)=0, \xi_t(y_k)=1\right)\\
= & \sum_{x,y_k\sim x}p_kP\left( \xi_t(x)=0,  \xi_t(y_k)=1\right)\\
=& \sum_{y_k, x \sim y_k}p_kP\left( \xi_t(x)=0,  \xi_t(y_k)=1 \right)\\
= & \sum_{y_k, x \sim y_k} \frac{p_k}{d_k(y_k)}\sum_{y_k*(k-1)\neq x}P\left( \xi_t(x)=0,  \xi_t(y_k)=1\right)
\end{align*}
Interchanging the role of $x$ and $y_k$, the above
\begin{equation*}
= \sum_{x, y_k\sim x} \frac{p_k}{d_k(x)}\sum_{x*(k-1)\neq y_k}P\left( \xi_t(x)=1, \ \xi_t(y_k)=0\right)
\end{equation*}
Then \eqref{step1} can be reformulated as 
\begin{align*}
&=-\sum_x \frac{p_k}{d_k(x)}\sum_{x*k}P\left( \xi_t(x)=1, \xi_t(y_i)=0 \hbox{ for all $1\le i \le k$}\right)\\
&+\sum_{x,y_k\sim x} \frac{p_k}{d_k(x)}\sum_{x*(k-1)\neq y_k}P\left( \xi_t(x)=1, \ \xi_t(y_k)=0\right)\\
&+\sum_{x,y_k\sim x} \frac{p_k}{d_k(x)}\sum_{x*(k-1)\neq y_k}P(\xi_t(x)=\xi_t(y_k)=0,  \xi_t(y_i)=1\ \hbox{for some }\ i<k)
\end{align*}
Combining the first two summations and summing over $k>1$ gives the desired result.
\end{proof}

\subsection{Step 2: Frontier lower bounds}

Pick a vertex from the tree to be the root and call it $x_0$. Given a vertex $x$ in the tree we say that
$x'$ is a child of $x$ if it is a neighbor of $x$ and further away from the root than $x$ is. 
We define the {\bf subtree generated by} $x'$, $S(x')$, to be all of the vertices that can be reached from $x'$ without going through $x$. By definition, $x' \in S(x')$. For any finite set on the tree $A$, define its {\bf frontier} $F(A)$ as the set of sites $x\in A$ that have a child $x'$ such that the subtree $S(x')\cap A=\emptyset$ and define the {\bf exterior boundary of A}, $H(A)$ to be the set of all such children $x'$. That is, $x'\in H(A)$ if and only if $S(x')\cap A=\emptyset$ and the parent of $x'$ is in $F(A)$. to help visualize the definitions, see Figure \ref{fig:FAdef}. Our next step is to lower bound the sizes of the sets we just defined. 

\begin{lemma}\label{lem1}
$|H(A)|\geq |A|$ and $|F(A)|\geq |A|/(M-1)$.
\end{lemma}

\begin{proof}
We prove the first result by induction on the cardinality of $|A|$. If $|A|=1$, the result is trivial as $|H(A)|\geq d(x)-1\geq 2$. Suppose now that the result is true for all $B$ with $|B|\leq n-1$ and let $|A|=n$. Let $x\in A$ be the point with the largest distance to the root and let $B= A\setminus \{x\}$. Then by induction $|H(B)|\geq n-1$. Since none of the descendents of $x$ are in $A$, but $x$ might be in $H(B)$.
$$
|H(A)| \geq |H(B)| - 1 + d(x) - 1 \geq (n-1)-1 + 2=n
$$
The second result follows from the first since $|H(A)|\leq (M-1)|F(A)|$.
\end{proof}

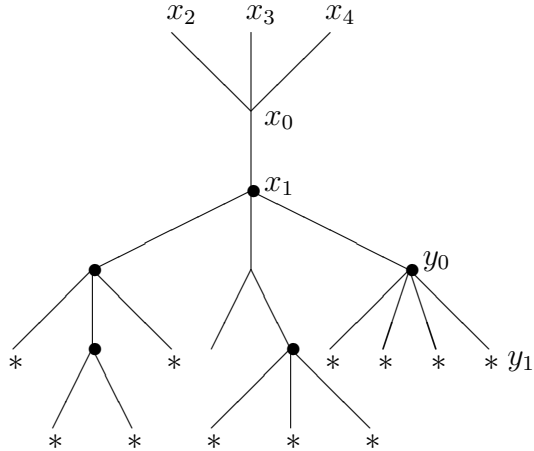
\begin{figure}[ht]

\begin{center}
\begin{picture}(240,210)
\put(120,150){\line(1,1){30}}
\put(120,150){\line(-1,1){30}}
\put(120,150){\line(0,1){30}}
\put(120,150){\line(0,-1){30}}
\put(120,120){\line(0,-1){30}}
\put(120,120){\line(-2,-1){60}}
\put(120,120){\line(2,-1){60}}
\put(120,90){\line(-1,-2){15}}
\put(120,90){\line(1,-2){15}}
\put(180,90){\line(-1,-3){10}}
\put(180,90){\line(1,-3){10}}
\put(180,90){\line(-1,-1){30}}
\put(180,90){\line(1,-1){30}}
\put(60,90){\line(1,-1){30}}
\put(60,90){\line(-1,-1){30}}
\put(60,90){\line(0,-1){30}}
\put(60,60){\line(-1,-2){15}}
\put(60,60){\line(1,-2){15}}
\put(135,60){\line(1,-1){30}}
\put(135,60){\line(-1,-1){30}}
\put(135,60){\line(0,-1){30}}
\put(125,145){$x_0$}
\put(125,120){$x_1$}
\put(88,185){$x_2$}
\put(118,185){$x_3$}
\put(148,185){$x_4$}
\put(118,117){$\bullet$}
\put(58,87){$\bullet$}
\put(178,87){$\bullet$}
\put(185,92){$y_0$}
\put(58,57){$\bullet$}
\put(133,57){$\bullet$}
\put(28,52){$\ast$}
\put(88,52){$\ast$}
\put(148,52){$\ast$}
\put(168,52){$\ast$}
\put(188,52){$\ast$}
\put(208,52){$\ast$}
\put(217,54){$y_1$}
\put(43,22){$\ast$}
\put(73,22){$\ast$}
\put(103,22){$\ast$}
\put(133,22){$\ast$}
\put(163,22){$\ast$}
\end{picture}
\caption{For simplicity we have only drawn the edges from vertices within distance 3 of the root
that are relevant to the definitions. $\bullet$ indicates sites in $A$. 
All the $\bullet$s are in $F(A)$r except for $x_1$.
$\ast$s mark the points in $H(A)$.}
\end{center}
\label{fig:FAdef}
\end{figure}

\subsection{ODE lower bounds}

\noindent Let $A_t = \{ x : \xi_t(x)=1\}$ Our next step is

\begin{lemma} \label{Atbd}
Let $\gamma = -p_0 + \sum_k p_k (k-1)$ (which is $>0$ by assumption).
$$
\frac{d}{dt} E|A_t| \ge \gamma E|A_t|
$$
\end{lemma}

\begin{proof}
Let $l_x$ be the distance of $x$ from the root. 
The expression on the second line in \eqref{ode} is $\geq 0$. The third line is
\begin{align*}
&= \sum_{x,y_k\sim x} \frac{p_k}{d_k(x)}\sum_{x*(k-1)\neq y_k} P(\xi_t(x)=\xi_t(y_k)=0, \hbox{ some } \xi_t(y_i)=1)
\\
&\geq E\sum_{x\in H(A_t),l_{y_k}>l_x} \frac{p_k}{d_k(x)}\sum_{y_i\in F(A_t)\ for\  some\ 1\leq i\leq k-1} 1
\\
&= E\sum_{x\in H(A_t),l_{y_k}>l_x} \frac{p_k}{d_k(x)}\cdot (d(x)-1) \cdot (k-1) \cdot \binom{d(x)-2}{k-1}
\\
&=(k-1)p_k E|H(A_t)| \ge (k-1)p_k |A_t|
\end{align*}
In the third line, $d(x)-1$ gives the choices for $y_k$. $k-1$ is because we have $k-1$ choices from $y_1,...,y_{k-1}$ to be on the frontier. Suppose $y_1$ is chosen to be in the frontier, then the number of choices for $y_2, \ldots y_{k-1}$ is ${d(x)-2 \choose k-1}$. The final inequality comes from Lemma \ref{lem1}
\end{proof}

Choose a neighbor $x_1$ of the root $x_0$. (See Figure \ref{fig:FAdef} for a picture.)  Set all the sites outside of $S(x_1)\cup\{x_0\}$ to be always equal to 0. Let $\bar{\xi}_t$ be the process restricted to $S_1 \equiv S(x_1)\cup \{x_0\}$. Let 
\begin{align*}
\bar{A}_t = \{ x : \bar\xi_t(x)=1 \} & \qquad A^*_t = \bar{A}_t \cap S(x_1) \\
H^*(\bar{A}_t) = H(\bar{A}_t) \cap S(x_1)& \qquad  F^*(\bar{A}_t) = F(\bar{A}_t) \cap S(x_1)
\end{align*} 

\begin{lemma} \label{barAtbd}
$$
\frac{d}{dt} E|\bar A_t| \ge \gamma E|\bar A_t| - (\gamma+1)(M-1)
$$
\end{lemma}

\begin{proof} We repeat the proof of Lemma \ref{Atbd}. The differential equation in \eqref{step0} remains valid but 
when we make the transition to \eqref{step1} there is a term with $k=1$ that does not cancel:
$$
- p_1[d(x_0)-1]P(\bar{\xi}_t(x_0)=1)
$$
Note that if $x_0\in \bar{A}_t$, then $x_2,\dots, x_{d(x_0)}\in H(\bar{A}_t)$ so
$$
|H^*(\bar{A}_t)|\geq |H(\bar{A}_t)|-[d(x_0)-1] \geq |\bar A_t| - (d(x_0)-1)
$$
where the last inequality follows from  Lemma \ref{lem1}. Using $d(x_0) \le M$ the desired result follows. 
\end{proof}

Let $L=2(\gamma+1)(M-1)/\gamma$. Lemma \ref{barAtbd} implies that once $E|\bar A_t| \ge L$ it grows exponentially
with rate $\ge \alpha = (\gamma+1)(M-1)$. 

\begin{lemma}\label{lem2} 
There exists $\ep_0>0$ such that 
\begin{equation}
P( |\bar A_1| \ge L ) \ge \ep_0
\end{equation}
for all trees $\TT$ with $3 \le d_{min} \le d(x) \le M$
\end{lemma}

\begin{proof} Let $G_L$ be the event that at time 1 there is an occupied path from the root $x_0$ to distance $L-1$.
It is easy to see that there exists $\ep_0>0$ such that $P( G_L ) \ge \ep_0$
for all trees $\TT$ with $3 \le d_{min} \le d(x) \le M$. To see this note that the worst case occurs when all sties have degree 3 but
offspring are sent across an edge with probability $1/M$. 
\end{proof}

\begin{lemma}\label{lem3} 
There exists $t_0>0$ such that 
\begin{equation}
E|F^*(\bar A_{t_0})| \ge 2
\end{equation}
for all trees with $3 \le d_{min} \le d(x) \le M$
\end{lemma}

\begin{proof}
Conditioning on $\{ |\bar A_t| \ge L \}$ it follows that for all trees with $3 \le d_{min} \le d(x) \le M$.
$$ 
E|\bar{A}_t|\geq \ep_0 e^{\alpha(t-1)]}
$$
Now $|F^*(\bar A_t)| \ge |F(\bar A_t)| -1$ with equality if $x_0$ is in state 1, so by Lemma \ref{lem1}
$$
|F^*(\bar{A}_t)|\geq |F(\bar{A}_t)|-1 \geq \frac{1}{M-1}|\bar{A}_t|-1
$$
and the desired result follows.
\end{proof}

\subsection{Step 4: Lower bounding BRW} 

Now define a lower bounding branching random walk $Z_n$. Let $Z_0=\{ x_0\}$, where $x_0$ is the root. Let $Z_1=F^*(\bar{A}_{t_0})=F(\bar{A}_{t_0})\cap S(x_1)$. Inductively, given $Z_n$, note that for any $x\in Z_n$, $x$ has a child $x'$ such that $S(x')\cap\bar{A}_{nt_0}=\emptyset$.   Let $F^*(\bar{A}^{x,0}_{t_0})$ be the children of $x$, where the superscript $0$ means that to obtain $\bar{A}^{x,0}_t$, we enforce 0-boundary condition on sites above $x$. Hence all the neighbors of $x$ except for $x'$ are in state 0 during $[nt_0, \ (n+1)t_0]$. Therefore all $\bar{A}^{x,0}_t \ \forall x\in Z_n$ are independent and $E|F^*(\bar{A}^{x,0}_{t_0})|>1 \ \forall x\in Z_n$ by Lemma \ref{lem2}. Define the $n+1$ th generation by
\begin{equation*}
Z_{n+1}=\bigcup_{x\in Z_n}F^*(\bar{A}^{x,0}_{t_0})
\end{equation*}

\begin{lemma}\label{lem3}
There exists $C_v>0$ such that for all trees $\TT$ with $3 \le d(x) \le M$ for all $x$
\begin{align}
&E[Z_{n+1}|Z_n]\geq 2Z_n\\
&Var(Z_{n+1}|Z_n)\leq C_vZ_n\label{neq2}
\end{align}
\end{lemma}
\begin{proof}
Given any tree $\TT$, note that  $Z_{n+1}=\bigcup_{x\in Z_n}F^*(\bar{A}^x_{t_0})$ and $F^*(\bar{A}^x_{t_0})\cap F^*(\bar{A}^y_{t_0})=\emptyset$ if $x\neq y$. Then
\begin{equation}
E[Z_{n+1}|Z_n, \TT]=\sum_{x\in Z_n} E\left[ |F^*(\bar{A}^{x}_{t_0})| |Z_n, \ T\right]>(1+\epsilon)Z_n
\end{equation}
To prove (\ref{neq2}), let $\eta_t$ be a branching process where $\eta_0=1$ and every particle gives a birth at rate $M$ without death. Then given any tree, $|\bar{A}_t| $ is stochastically bounded by $|\eta_t|$.
So now, let $\TT^x$ denote the subtree consisting of $x$ and its descendents. By independence
\begin{align*}
&\var\left( Z_{n+1}|Z_n, \TT\right) =\sum_{x\in Z_n} \var\left(|\bar{A}^{x,0}_{t_0}| | \TT^x \right)\\
&\leq \sum_{x\in Z_n}E\left[| \bar{A}^{x,0}_{t_0}|^2\ |  \TT^x\right] \leq \sum_{x\in Z_n}E|\eta_{t_0}|^2 =C_v Z_n
\end{align*}
Since $T$ is arbitrary, we have completed the proof.
\end{proof}

The next Lemma completes the proof of Theorem \ref{genthm}.
\begin{lemma}
With positive probability
\begin{equation}
\liminf_{n\to \infty}\frac{Z_n}{(3/2)^n}\ge 1.
\end{equation}
\end{lemma}

\begin{proof}
First by Lemma \ref{lem3} and Chebyshev's Inequality,
\begin{align*}
&\ P\left(Z_{n+1}< (3/2)^{n+1} |Z_n\geq (3/2)^n\right)\\
&\leq  P\left( |Z_{n+1}-E[Z_{n+1}|Z_n]|> Z_n/2 |Z_n\geq (3/2)^n\right)\\
&\leq  E\left( \left. \frac{C_vZ_n}{(Z_n/2)^2} \right| Z_n\geq (3/2)^n\right) \leq 4C_v \cdot (2/3)^n \equiv \delta_n
\end{align*}
Pick $n_0$ large enough so that $\delta_{n_0}<1$. It follows from the proof of Lemma \ref{lem2} that $P( Z_{n_0}\geq (3/2)^{n_0}) > 0$.
Since $\delta_n$ is decreasing, we have
$$
P\left(\left. \liminf_{n\to \infty} \frac{Z_n}{(3/2)^n}\geq 1 \right| Z_{n_0}\geq (3/2)^{n_0}\right) \geq \prod^\infty_{n={n_0}}(1-\delta_n) > 0
$$
which proves the desired result.
\end{proof}

\clearp

\section{Results for $d$-regular trees}

\subsection{Extinction}

The first result is elementary but a proof is cincluded for completeness.

\begin{lemma} \label{nohit}
Let $h(x)$ be the probability two continuous time random walks separated by $x$ on a $d$-regular tree will hit. 
$$
h(x) = \left(\frac{1}{d-1}\right)^{x}
$$
\end{lemma}

\begin{proof} 
If the two particles are at distance $x>0$ then the probability they are at distance $x+1$ after the first jump is  $(d-1)/d$, while they are at distance $x-1$ with probability $1/d$.  
\begin{align*}
\left( \frac{1}{d-1}\right)^x &= \frac{d-1}{d} \left( \frac{1}{d-1}\right)^{x+1}+\frac{1}{d}\left(\frac{1}{d-1}\right)^{x-1} \\
&= \frac{1}{d} \left( \frac{1}{d-1}\right)^{x}+\frac{d-1}{d}\left(\frac{1}{d-1}\right)^{x} 
= \left( \frac{1}{d-1}\right)^{x}
\end{align*}
i.e., if $X_t$ is the distance between two coalescing random walks on a $d$-regular tree then  $((d-1)^{-X_t}$ is a martingale.
Since $h(0)=1$, $h(x)\le 1$ for $x\ge 0$ and $h(x)\to 0$ as $x\to \infty$ the desired result follows from the optional stopping theorem.
\end{proof}

Let $\beta$ be the probability two newborn particles in the dual do not coalesce. 
Since two newborn particles are at distance two from each other  
$$
\beta =  1 - \frac{1}{(d-1)^2}.
$$

\mn
{\bf Theorem \ref{ddieA}}
{\it On a d-regular tree the COBRA dies out if}
 \begin{equation*}
d\beta\sum_{k\geq 2}(k-1)p_k-p_0<0
\end{equation*}

\begin{proof}
In COBRA, a particle dies at rate $p_0$ and gives birth to $k$ particles at rate $dp_k$. To make the dual process more like a branching random walk, when a particle dies and gives birth to a positive number of particles, we declare that the particle did not die but jumped to the location of particle 1. If no offspring were produced then the particle dies. To get an upper bound on the growth of the dual (i) we ignore coalescence between the lineages that are not siblings, and (ii) if $k$ particles are born we ignore coalescence between particles $i>1$ and $j>1$. Note that particles $2, \ldots k$ each have probability $\ge \beta$ of not coalescing with 1. Thus the expected number of the particles that do not coalesce with 1 is $(k-1)\beta$. If we use $\eta^0_t$ to denote the resulting system 
starting from a single particle then 
$$
\frac{d}{dt} E\eta^0_t = \left[-p_0 + d\sum_k p_k(k-1)\beta \right] \, E\eta^0_t
$$
It is immediate that if $d\beta\sum_{k\geq 2}(k-1)p_k-p_0<0$ then
$E|\zeta^0_t| \le  E|\eta^0_t| \to 0$.
\end{proof}

\subsection{Local Survival}\label{lsvpf}

Recall that $\mu = \sum_k kp_k$ is the mean number of offspring in the dual.

\mn
{\bf Theorem \ref{ddlocal}}
{\it Given a d-regular tree $T$, the zealot voter model dies out locally if 
\beq
\mu<\frac{d(1-p_0)+p_0}{2\sqrt{d-1}}.
\label{lsdreg}
\eeq
If $p_0=0$ this is $\mu < d/(2\sqrt{d-1})$.}

\begin{proof}
Using a superscript 0 to denote the process starting with only the root occupied, we need to show 
\begin{equation}
P(\xi^0_t \cap \{0\}\ne\emptyset)\longrightarrow 0\quad\hbox{as $t\to \infty$.} 
\end{equation}
By duality,
\begin{equation}
P(\xi^{0}_t\cap \{0\}\neq\emptyset)=P(\zeta^0_t\cap \{0\}\neq \emptyset)
\end{equation}
Let $\eta^0_t\supset \zeta^0_t$ be the BRW in which particles die at rate $p_0$ and at rate $dp_k$ die and give birth onto $k$ neighbors chosen without replacement.

\begin{lemma}
Let $m(t,x)=E\eta^0_t(x)$ be the expected number of particles on site $x$ at time $t$. Then $m(t,x)$ satisfies the equation
$$
\frac{d}{dt}m(t,x)=-\alpha m(t,x)+\sum_{y\sim x} m(t,y)\mu \quad\hbox{where $\alpha = d(1-p_0)+p_0$}
$$
The solution is given by 
\beq
m(t,x)=e^{(\mu-\alpha)dt}P(S^0_t=x)
\label{slxn}
\eeq
where $S^0_t$ is the random walk on tree $T$ starting from the root that jumps at rate $d\mu$ to a neighbor chosen uniformly at random.
\end{lemma}

\begin{proof}
To check (\ref{slxn}), note that using RHS for the right-hand side of the equation
\begin{align*}
\frac{d}{dt} RHS
&=(\mu-\alpha)d e^{(\mu-\alpha)dt}P(S^0_t=x)+e^{(\mu-\alpha)dt}\left[-d\mu P(S^0_t=x)+\sum_{y\sim x}d\mu\times\frac{1}{d} P(S^0_t=y)\right]\\
&=-\alpha d e^{(\mu-\alpha)dt}P(S^0_t=x)+\sum_{y\sim x} \mu e^{(\mu-\alpha)dt}P(S^0_t=y)\\
&=-\alpha m(t,x)+\sum_{y\sim x} m(t,y)\mu
\end{align*}
which gives the desired result. 
\end{proof}

Let $X_t=|S^0_t|$ be the distance from the root. 
We couple $X_t$ to a simple random walk $\hat{X}_t$ on $\mathbb{Z}$ that jumps to the left at rate $\mu$ and to the right at rate $(d-1)\mu$ by using the following recipe: $\hat{X}_t$ follows the move of $X_t$ if $X_t\ne 0$; when $X_t$ jumps from 0 to 1, $\hat{X}_t$ jumps to the left with probability $1/d$. Clearly,
\begin{equation*}
\hat{X}_t \leq X_t \hspace{1cm} \forall t\geq 0
\end{equation*}
and hence
\beq
P(S^0_t=0)=P(X_t=0)\leq P(\hat{X}_t \leq 0)
\label{mono}
\eeq
Note that if $\theta \le 0$ then
\begin{align*}
P(\hat{X}_t\leq 0)&\leq Ee^{\theta \hat{X}_t}
=\sum_{k=0}^\infty e^{-d\mu t}\cdot \frac{(d\mu t)^k}{k!}\left(\frac{1}{d}e^{-\theta}+\frac{d-1}{d}e^{\theta}\right)^k\\
&= \exp\left\{ -d\mu t\left[1-\left(\frac{1}{d}e^{-\theta}+\frac{d-1}{d}e^\theta\right)\right]\right\} \\
&= \exp\left\{ -\mu t[d-(e^{-\theta}+ (d-1)e^\theta)]\right\} 
\end{align*}
To optimize this bound we maximize the term in square brackets.  To do this, we set 
$$
0 = \frac{d}{d\theta} [d-(e^{-\theta}+ (d-1)e^\theta)] = e^{-\theta}-(d-1)e^\theta
$$
Solving we have $e^{2\theta} = 1/(d-1)$ or $e^{\theta} = 1/\sqrt{d-1}$, which leads to the bound
$$ 
P(\hat{X}_t\leq 0) \leq \exp\left\{ -(d-2\sqrt{d-1})\mu t\right\}
$$
Using this with \eqref{slxn} and \eqref{mono} we have  
\begin{align*}
m(t,0)&=e^{(\mu-\alpha)dt}P(S^0_t=0)\\
&\le \exp\left\{\left[\left(d-(d-2\sqrt{d-1})\right)\mu-d\alpha\right]t\right\}\\
&= \exp\left\{\left(2\sqrt{d-1}\mu-d\alpha\right)t\right\}
\end{align*}
Since $\alpha=p_0+d(1-p_0)$ our assumption \eqref{lsdreg} implies the exponent is negative. We have completed the proof.
\end{proof}

\mn
{\bf Theorem \ref{dslocalbbd}.} {\it On a $d$-regular tree the zealot voter model survives locally if }
$$
p_0=0\quad\hbox{and}\quad \mu>\frac{d}{\sqrt{d-1}+1}.
$$

\begin{proof}
Choose a self-avoiding path 
$\{e_n,-\infty<n<\infty\}$ in $T^d$  such that $e_0=0$ is the root and $|e_n-e_{n+1}|=1$. This gives an embedding of $\mathbb{Z}$ into $T^d$. 
Now define 
\begin{equation*}
u(n)=P\left( e_n\in \zeta_t\ \text{for some} \ t\right)
\end{equation*}
for $n\ge 0$. By the strong Markov property, for all $n.m\ge 0$
\begin{equation*}
u(n+m)\ge u(n)u(m) 
\end{equation*}
i.e., the sequence is supermultiplicative. This implies that 
$$
\beta(\mu)\equiv \lim_{n\to \infty}[u(n)]^{1/n} = \sup_{m\ge 1} [u(m)]^{1/m}.
$$

Let $S(e_0)$ denote the subtree starting from $e_0$ that does not include $e_{-1}$. Consider a lower bound $\bar\zeta_t$ on the dual COBRA where we birth are only allowed in $S(e_0)$. Our next step is to state a result from the contact process. This is Lemma 4.53 in \cite{L99} but the proof also works for our COBRA.
\begin{lemma}\label{betamsty}
\begin{equation}
\displaystyle\lim_{n\to \infty}\left[\sup_t P(e_n\in \bar{\zeta}_t)\right]^{1/n}=\beta(\mu)
\end{equation}
\end{lemma}

\noindent Since $\bar{\zeta_t}\subset \zeta_t$, the desired result follows from the next two Lemmas.

\begin{lemma}\label{lsv1}
If $\beta(\mu)>1/\sqrt{d-1}$, then $\displaystyle\inf_t P\left( e_0 \in \bar{\zeta_t}\right)>0$.
\end{lemma}

\begin{lemma}\label{muontr}
If $\mu>d/(\sqrt{d-1}+1)$, then $\beta(\mu)>1/\sqrt{d-1}$.
\end{lemma}
\end{proof}

\begin{proof}[Proof of Lemma \ref{lsv1}]
The proof here is almost identical to the one on pages 99-100 of Liggestt \cite{L99}. 
According to Lemma \ref{betamsty} and our assumption, we can fix constants $a>1/\sqrt{d-1}$, $n\ge1$ and $s>0$ such that
\begin{equation}\label{sbp}
P\left(e_n\in \bar{\zeta}_s \right)=a^n
\end{equation}
We now follow the proof of Proposition 4.57 in Liggett \cite{L99} to construct an embedded branching process. Let $B_0=\{e\}$ and $B_1=\{x\in \bar{\zeta}_s: |x-e|=n\}$. We ignore all the births outside $S(x)$ and apply the same rules leading from $B_0$ to $B_1$ to obtain a random subset $B(x)$ of $\{y\in S(x)\cap \bar{\zeta}_{2s}: |y-e|=2n\}$. Let $B_2=\cup_{x\in B_1}B(x)$. We repeat the same rule to construct a branching process $B_j$. Note $B_j\subset \bar{\zeta}_{js}$. Moreover $B_j$ is supercritical since by (\ref{sbp}) the offspring distribution has mean $(d-1)^na^n>1$. Then 
\begin{equation*}
\displaystyle\lim_{j\to\infty}\frac{|B_j|}{\left((d-1)^na^n\right)^j}
\end{equation*}
exists and is positive with positive probability. As a result, we can find an $\epsilon$ such that for all sufficiently large $j$,
\begin{equation*}
P\left(|B_j|>\epsilon \left((d-1)a\right)^{nj}\right)>\epsilon
\end{equation*}
We will show particles from the subchain $\{B_{ji}\}^\infty_{i=0}$'s are sufficient to make the process survive locally. Since it takes time $ijs$  to get to $B_{ji}$, we let
\begin{equation}
r_i=P(0\in \bar{\zeta}_{2ijs})
\end{equation}
It follows from the strong Markov property that
\begin{equation}\label{riter}
r_{i+1}\ge P(x\in\bar{\zeta}_{2(i+1)js}\ {\rm for\ some} \ x\in B_{j})P(e_{nj}\in \bar{\zeta}_{js})
\end{equation}
Let $\lfloor y\rfloor$ be the largest integer $\le y$ and let $N=\lfloor \epsilon((d-1)a)^{nj}\rfloor$. This is
\begin{align}
&\ge P(|B_j|>N)[1-(1-r_i)^N]P(e_{nj}\in \bar{\zeta}_{js})
\nonumber\\
&\ge \epsilon [1-(1-r_i)^N]P(e_{nj}\in \bar{\zeta}_{js})\label{interiter1}
\end{align}
Using the strong Markov property on the last probability gives
\begin{equation}\label{interiter2}
P(e_{nj}\in\bar{\zeta}_{js})\ge \left[P(e_n\in \bar{\zeta}_s)\right]^j=a^{nj}
\end{equation}

\mn
Let $f(r)=\epsilon[1-(1-r)^N]a^{nj}$. Combining (\ref{interiter1}), (\ref{interiter2}) and (\ref{riter}) gives
\begin{equation*}
r_{i+1}\ge f(r_i)
\end{equation*}
Note $f(r)$ is increasing over $[0,1]$ with $f(0)=0$. Moreover, $f'(r)=\epsilon a^{nj}N(1-r)^{N-1}$.  So using the definition of $N$,
\begin{align*}
f'(0)&=\epsilon a^{nj}N\\
&\ge\epsilon a^{nj}[\epsilon ((d-1)a)^{nj}-1]\\
&=\epsilon^2 [a^{2}(d-1)]^{nj}-\epsilon a^{nj}
\end{align*}
Since $a>1/\sqrt{d-1}$ this is $>1$ if $j$ is chosen large enough. Thus $f(r)$ has a fixed point $r^*\in (0,1]$. We will prove by induction that 
\begin{equation}\label{rge0}
r_i\ge r^*, \ \forall i
\end{equation}
\mn
When $i=0$, the inequality is trivial since $r_0=1$. Suppose $r_i\ge r^*$. By the monotonicity of $f(r)$, we have 
\begin{equation*}
r_{i+1}\ge f(r_i)\ge f(r^*)=r^*
\end{equation*}
To generalize (\ref{rge0}) to all time $t$. Note particles die at rate $d$. Precisely
$$
P(e\in \bar{\zeta}_t|e\in\bar{\zeta}_{2ijs})\ge e^{-d(t-2ijs)}
$$
In particular
\begin{equation*}
P(e\in \bar{\zeta}_t)\ge e^{-djs}r_i, \ \text{if}\ 2ijs<t<2(i+1)js
\end{equation*}
We have completed the proof.
\end{proof}

\begin{proof}[Proof of Lemma \ref{muontr}]
Consider a simple random walk on $\mathbb{Z}$ which takes steps
$$
\begin{cases}
+1 &\text{with probability}\ \frac{d-\mu}{d}\\
-1 &\text{with probability}\  \frac{\mu}{d}
\end{cases}
$$
Repeating the proof of Lemma \ref{nohit} shows that $\phi(x)=\left(\frac{\mu}{d-\mu}\right)^x$ is a martingale. The stopping time theorem for martingales shows
\begin{equation*}
P_n(T_0<\infty)=\left(\frac{\mu}{d-\mu}\right)^n
\end{equation*}
Note
\begin{equation*}
P(e_n\in \zeta_t \ \text{for some} \ t>0)\ge P_{e_n}(T_e<\infty) \ge \left(\frac{\mu}{d-\mu}\right)^n\end{equation*}
where the second one is the probability that the COBRA initiated at $e_n$ ever visits the root. Then
\begin{equation*}
[u(n)]^{\frac{1}{n}}\ge \frac{\mu}{d-\mu}
\end{equation*}
Since $\mu>\frac{d}{\sqrt{d-1}+1}$, by assumption we have 
\begin{align*}
\beta(\mu)&=\lim_{n\to \infty}[\mu(n)]^{1/n} \ge \frac{d}{d-\mu}-1\\
&> \frac{d}{d-d/(\sqrt{d-1}+1)}-1=\frac{1}{\sqrt{d-1}}
\end{align*}
which completes the proof of Lemma \ref{muontr} and hence the proof of Theorem \ref{dslocalbbd}.
\end{proof}

\clearp

\section{Results for Galton-Watson Trees} \label{GWsec}

\subsection{Survival of COBRA}

We will now prove Theorem \ref{prtbxn}. To lead up to that we will describe the proof of the main result in \cite{DZ} in dimensions $d\ge 3$.
The model under consideration there is a biased voter model with small bias. Jumps at $x$ from $0 \to 1$ occur st rate $(1+\ep)f_1(x)$, where
$f_i(x)$ is the fraction of neighbors in state $i$, while jumps from $1 \to 0$ occur at rate $f_0(x)$. Suppose for concreteness that
the neighborhood consists of the $2d$ nearest neighbors. As in the case of the zealot voter model the process is additive in the sense of Harris \cite{H76}
and can be constructed from a graphical representation with independent Poisson processes, $T^{x,i}_n$, $n \ge 1$ for $i=1,2$. Let $e_1, \ldots, e_{2d}$
be an enumeration of the nearest neighbors of 0.

\begin{itemize}
  \item The $T^{x,1}_n$ have rate 1 and have associated independent random variables $U^{x,1}_n$ that are uniform on $\{ 1, 2, \ldots 2d \}$.
At time $T^{x,i}_n$ we write a $\delta$ at $x$ that will kill a 1 at $x$ and draw an arrow from $x+e(U^{x,i}_n)$ to $x$. By considering the four
cases for the states of $x+e(U^{x,i}_n)$ and $x$ we can easily check that this gadget causes $x$ to imitate its neighbor. 
\item The $T^{x,2}_n$ have rate $\ep$ and have associated independent random variables $U^{x,2}_n$ that are uniform on $\{ 1, 2, \ldots 2d \}$.
At time $T^{x,i}_n$ we draw an arrow from $x+e(U^{x,i}_n)$ to $x$ which will cause $x$ to be 1 if $x+e(U^{x,i}_n)$ is.
\end{itemize} 

Since branching occurs at rate $\ep$ in dual, the suggests that we should run time at rate $1/\ep$ and scale space by $1/\sqrt{\ep}$ to mkae the dual process
converge to a branching Brownian motion. One complication is that new born particles will coalesce with their parent with a probability $\gamma$ which is the probability a random walk started at $e_1$ returns to 0. It is not hard to show that the probability such a coalescence will occur after time $1/\sqrt{\ep}$ 
tends to 0. Thus in order for the sequence of processes to be tight, we do not add the newly born particle until time $1/\sqrt{\ep}$ has elapsed. Other estimates in the proof show that it is unlikely for particle to coalesce with another particle that is not its parent, so the sequence of rescaled processes
converges to a branching random walk in which new particle are born at rate $\gamma$.

In \cite{DZ} this observation is combined with a block construction to prove the existence of a stationary distribution in a ``hybrid zone'' in which the process on $x_1 \ge 0$ is a biased voter model that favors 1 and on $x_1<0$ the process is a biased voter model favoring 0. Things are simpler for the zeaalot voter model on trees. If we only want to prove survival of the dual it is enough to prove that when time is run at rate $1/\ep$
the size of the dual converges to a supercritical branching process. Taking into account the fraction of time a random walk spends at vertices of
degree $k$ we arrive at:

\mn
{\bf Theorem 7.} {\it Let $\delta>0$. If $\ep>0$ is small enough then the COBRA dies out if
$$
\sum_k p_k (\mu_{m,k} - 1) - p_0 < -\delta
$$
and survives if the last quantity is $> \delta$.}

\subsection{Local Survival}

\subsubsection{Proof of Theorem \ref{lcdout}}

Since the branching random walk gives an upper bound for the COBRA, it suffices to show 

\begin{lemma}\label{LSbrw}
If $p_0=0$, then on a d-regular tree, the threshold for the local survival of $\eta^0_t$ satisfies
$$
\mu_l(\mathcal{T}^d)=d/(2\sqrt{d-1})
$$
where $\eta^0_t$ is the branching random walk starting with 1 particle at the root
\end{lemma}
To prove this result, define $M(v,\ n)$ to be the number of oriented loops of length $n$ starting from vertex $v$. It is well-known that the limit $L=\displaystyle\lim_{n\to \infty}M(v,\ 2n)^{1/2n}=\displaystyle\sup_{n\to \infty}M(v,\ 2n)^{1/2n}$ exists for all graphs, independent of the choice of vertex. This follows from a simple supermultiplicativity argument. Furthermore, define an  {\it evolutionary walk } from vertex $u$ to vertex $v$ to be a sequence $0\le T_{n_0}^{x_0,i_0}<T_{n_1}^{x_1,i_1}<\dots<T_{n_m}^{x_m, i_m}<\infty$ with $x_0=u, \ x_m=v$. Precisely, this corresponds to a path in the graphical representation such that the fluid can flow from $u$ to $v$. By definition, for a fixed path of length $n$ on the tree, the expected number of evolutionary walks on this path is $(\mu/d)^n$. This is because when a branching event occurs, the expected number of particles landing on a certain neighbor is $\mu/d$. We will show

\begin{lemma}
Suppose $L=\displaystyle\lim_{n\to \infty}M(v,\ 2n)^{1/2n}$. Then $\mu_l(\mathcal{T}^d)=d/L$.
\end{lemma}
\begin{proof}
Let $X_n$ be the number of evolutionary walks of length $n$ starting and ending at the root $e_0$. Note $\{X_{2nk}\}$ dominates a branching process with offspring distribution given by $X_n$. In particular,
$$
EX_n\ge (\mu/d)^{2n}M(e_0,2n)
$$

So if $\mu>d/L$, this branching process is supercritical if $n$ is sufficiently large. Choose $n$ so that the above expectation is $>1$. Note $\forall T>0$, the expected number evolutionary walks of length $n$ by time $T$ is 
$$
\le \Gamma (d,n)=\frac{d^n}{(n-1)!}\int^T_0 e^{-s}s^{n-1}ds\le\frac{(dT)^n}{n!}
$$
The $\Gamma(d,n)$ comes from a sum of $n$ exponential distributions with parameter $d$. Note this is summable with respect to $n$. By Borel-Cantelli Lemma, the maximal length of evolutionary walks within any finite time $T$ is bounded and thus the root has to be visited infinitely often. For the other direction, note the expected number of evolutionary walks traversing $e_0$ is bounded by
$$
\sum^\infty_{n=1} (\mu/d)^nM(e_0,n)<\infty,\  {\rm if}\ \mu<d/L
$$
The proof is complete.
\end{proof}

\begin{proof}[Proof of Lemma \ref{LSbrw}]
It remains to compute $L$. This is given by Pemantle and Stacey \cite{PemSta}. To summarize, note for an oriented loop of length $2n$, $n$ steps are up (i.e. closer to $e_0$) $n$ steps are down. At each step, there are $d-1$ choices to move farther away. Hence
$$
M(0,2n)=\binom{2n}{n} (d-1)^n
$$
Use Stirling formula the desired result follows.
\end{proof}

\subsubsection{Condition for Local Survival}
We will now prove Theorem \ref{GWls}. In what follows, assume $p_0=0$. Let $p_x=\mu/d(x)$ be the probability that the particle at $x$ will be replaced by a child moving closer to the root. Define a harmonic function depending on the distance to the root by
\begin{equation}\label{mtg1}
\phi(x)=p_x\phi(x-1)+(1-p_x)\phi(x+1)
\end{equation}
Note (\ref{mtg1}) is equivalent to

\begin{align*}
\phi(x+1)-\phi(x)&=\frac{p_x}{1-p_x}\left[ \phi(x)-\phi(x-1)\right]\\
&=\frac{\mu}{d(x)-\mu}\left[ \phi(x)-\phi(x-1)\right]
\end{align*}
Suppose $0=x_0, x_1, \dots, x_n=x$ is the path from the root to $x$. We have 
\begin{align}\label{rcuxn}
\phi(x_n)-\phi(x_{n-1})&=\prod^{n-1}_{k=1}\frac{\mu}{d(x_k)-\mu}\left[\phi(x_1)-\phi(0)\right]
\end{align}
This recursion allows us to impose function $\phi(x)$ on each vertex $x\in G(V,E)$. By Theorem 6.4.8 in \cite{PTE} , if $\phi(x)\to \infty $ for all $l_x\to \infty$ then the dual survives locally. However, it is more convenient to pursue conditions such that the $\log$ increment $\log[\phi(x)-\phi(x-1)]\to\infty$ instead 
as we will see later.

Taking $\log$ of the recursion formula (\ref{rcuxn}) gives
\begin{equation*}
\log\left[ \phi(x_n)-\phi(x_{n-1})\right]=\log\left[\phi(x_1)-\phi(0)\right]+\sum^{n-1}_{k=1}\log\frac{\mu}{d(x_k)-\mu}
\end{equation*}

Suppose the Galton-Watson tree has degree distribution $\{q_j\}$. Now consider a branching random walk on $\mathbb{R}$ which has an initial particle at the origin. With probability $q_j$, it gives birth to $j-1$ particles at $\log \frac{\mu}{j-\mu}$ and this forms a point process $Z$. The location of the first generation is denoted as $\{z^1_r\}$ where $r$ is the index of each individual. For each particle $x$ in the first generation, it generates new particles in a similar way. The location of its children has the same distribution as $\{z^1_r+x\}$. We obtain the second generation by taking all the children of the first generation. Let $\{z^2_r\}$ be the locations of the second generation. The following generations are produced under the same manner. Denote $\{z^n_r\}$ as the location of the $n$th generation individuals. Let $F(t)=E\left[Z(-\infty,t]\right]$ be the expected number of points in $Z$ to the left of $t$. Define
\begin{align*}
m(\theta)&=\int^\infty_{-\infty}e^{-\theta t}dF(t)=\sum_{j\ge 3}q_j(j-1)\left(\frac{j-\mu}{\mu}\right)^\theta
\end{align*}
To avoid notational confusion, we use $\nu(a)=\inf\{e^{\theta a}m(\theta):\theta\ge 0\}$. This is (2.1) defined in \cite{B77}. It follows from Corollary (3.4) in \cite{B77} that $\nu(0)<1$ implies $\log[\phi(x)-\phi(x-1)]\to \infty$ for all $l_x\to \infty$. Hence $\nu(0)<1$ is a sufficient condition for local survival.

\noindent{\it Remark.} On the $d-$regular tree, 
\begin{equation*}
m(\theta)=(d-1)\left(\frac{d-\mu}{\mu}\right)^\theta
\end{equation*}
Then $v(0)<1$ iff $\mu>d/2$, which gives another proof of Theorem \ref{qpf}.

\subsection{Degree = 3 and 4}\label{3n4}
\mn
Our recursion is 
$$
\phi(x+)-\phi(x) = \frac{\mu}{d(x)-\mu} (\phi(x) - \phi(x-))
$$
$x-$ is neighbor closer to root. $x^+$ is any neighbor further away
\begin{align}
m(\theta)& =2q_3\left(\frac{3-\mu}{\mu}\right)^\theta+3q_4\left(\frac{4-\mu}{\mu}\right)^\theta
\label{mdef} \\
m'(\theta) &= 2q_3\left(\frac{3-\mu}{\mu}\right)^\theta\log\left(\frac{3-\mu}{\mu}\right) 
+3q_4\left(\frac{4-\mu}{\mu}\right)^\theta\log\left(\frac{4-\mu}{\mu}\right)
\label{mderiv}
\end{align}
Reacall $\nu(0) = \min\{ m(\theta) : \theta \ge 0 \}$.

\subsubsection{$\mu>2$}

Since $\mu/(3-\mu)$ and $\mu/(4-\mu)$ are both $> 1$, $\phi(x_n) \to \infty$ along any path $x_n \to \infty$,
so the process survives strongly.  

\subsubsection{$\mu\le 3/2$}

Since $\mu/(3-\mu)$ and $\mu/(4-\mu)$ are both $< 1$, $\phi(x_n) \to 0$ along any path $x_n \to \infty$.
However this only tells us that the proof fails.

\subsubsection{$3/2<\mu\le 2$}

\mn
Case 1. Note that if $2q_3>1$ there is a path to $\infty$ (which may not start at the root) along which we take the products of $\mu/(3-\mu)$
and hence $\phi(x_n) \to 0$, so the proof fails.

\mn
Case 2. $(4-\mu)/\mu > (3-\mu)/\mu$ so if
$$
m'(0) = 2q_3\log\left(\frac{3-\mu}{\mu}\right) +3q_4\log\left(\frac{4-\mu}{\mu}\right) > 0
$$
then $m'(\theta)>0$ for all $\theta>0$ and the minimum occurs at 0. $m(0) = 2q_3 + 3q_4 \ge 2$, 
so again the proof fails. let $q_3=p$ and $q_4=1-p$. For fixed $\mu$, $m'(0)$ is linear in $p$ so the condition holds when
$$
p < p_c = \frac{ 3 \log((4-\mu)/\mu)} { 3 \log((4-\mu)/\mu) + 2 \log(\mu/(3-\mu))}
$$
See Figure \ref{fig:m'(0)} for various $\mu$.

\begin{figure}[tbp] % float placement: (h)ere, page (t)op, page (b)ottom, other (p)age
  \centering
  % file name: C:/Users/Rick/Dropbox/Papers/ZVM/m'(0).pdf
  \includegraphics[bb=53 58 737 555,height=3.5in,keepaspectratio]{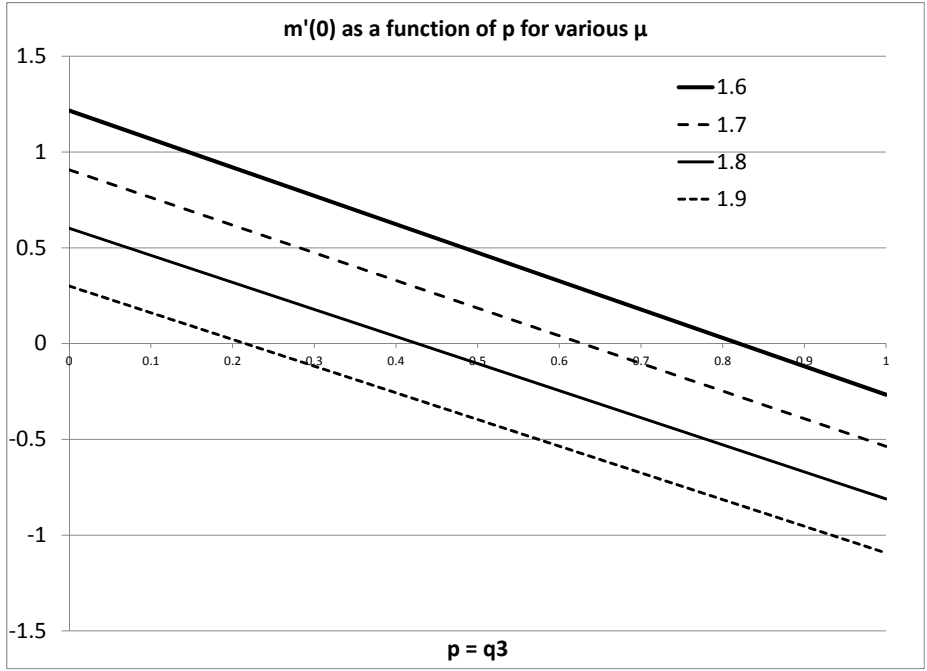}
  \caption{Local survival is possible only if $\mu'(0)<0$. This is $q_3>0.85$ for $\mu=1.6$; $q_3>0.65$ for $\mu=1.7$; $q_3>0.45$ for $\mu=1.8$ and $q_3>0.25$ for $\mu=1.9$.}
  \label{fig:m'(0)}
\end{figure}

\mn
Case 3. If $m'(0)<0$ then a minimum at $\bar\theta>0$ exists. Using \eqref{mderiv} we want
$$
2q_3 \left(\frac{3-\mu}{\mu}\right)^\theta\log\left(\frac{\mu}{3-\mu}\right)
=3q_4\left(\frac{4-\mu}{\mu}\right)^\theta\log\left(\frac{4-\mu}{\mu}\right)
$$
Cross multiplying
$$
\left( \frac{4-\mu}{3 - \mu}\right)^\theta = \frac{2q_3\log(\mu/(3-\mu))}{ 3q_4\log((4-\mu)/\mu) } 
$$
Let $A$ be the numerator and $B$ be the denominator of the fraction. $m'(0)<0$ implies $A>B$.
The LHS is 1 at $\theta=0$ and increases $\to\infty$ as $\theta\to\infty$ so a solution exists.
Taking logs
$$
\theta \log( (4-\mu)/(3-\mu) ) = \log(A) - \log(B)
$$
so we have
\beq
\bar\theta = \frac{\log(A)-\log(B)}{ \log( (4-\mu)/(3-\mu) ) }
\label{bartheta}
\eeq
There does not seem to be a good formula for $m(\bar\theta)$. To compute it numerically, we choose $\mu=1.6, \ 1.7, \ 1.8$ and  $1.9$ for
$$
\nu(0) = m(\bar\theta) = 2q_3 \exp( \bar\theta \log((3-\mu)/\mu ) + 3q_4 \exp( \bar\theta \log((4-\mu)/\mu )
$$
It shows from the table that the phase transition occurs at $q_3=0.996,\ 0.97, \ 0.91$ and $0.82$ respectively.

\begin{tabular}{l*{4}{c}}
$q_3$ & $\mu=1.6$ & $\mu=1.7$ & $\mu=1.8$ & $\mu=1.9$\\
\hline
0.8&	2.2	& 2.014149722	 &	1.597069414	&	1.074539921	\\
0.81 &	2.19		&	1.979137551	&	1.549560204	&	{\bf 1.030929768}	\\
0.82	& 2.179999993	&	1.942042353	&	1.500601354	&      {\bf 0.896331678}	\\
0.83	& 2.169035962	&	1.902724759	&	1.450116162	&	0.941977346	\\
0.88	& 2.079229445	&	1.666138794	&	1.171184237	&	0.708137684	\\
0.89	& 2.052259329	&	1.608953028	&	1.109102792	&	0.659079066	\\
0.9	& 2.021286727	&	1.547575636	&	{\bf 1.044453965}	&	0.609119574	\\
0.91	& 1.985646496	&	1.481425138	&	{\bf 0.976943138}	&	0.558174592	\\
0.92	& 1.944464056	&	1.409753116	&	0.906197761	&	0.506138592	\\
0.93	& 1.896552634	&	1.331568175	&	0.831734216	&	0.452876792	\\
0.94	& 1.840236437	&	1.245508443	&	0.752903513	&	0.398211752	\\
0.95	& 1.773023132	&	1.14961228	&	0.668796138	&	0.341900422	\\
0.96	& 1.690934707	&	{\bf 1.040865809}	&	0.578059943	&	0.283591365	\\
0.97	& 1.586938026	&	{\bf 0.914185487}	&	0.478505588	&	0.222735055	\\
0.98	& 1.446391322	&	0.759622966	&	0.366073412	&	0.158358633	\\
0.99	& 1.227510494	&	0.551306428	&	0.23102478	&	0.088284998	\\
0.995 &	{\bf 1.037752234}	\\											
0.996 &	{\bf 0.98268267}	\\											
0.997 &	0.915774891	\\											
0.998 &	0.828879261	\\

\end{tabular}

\end{document}